\documentclass[12pt]{amsart}
\usepackage{amssymb,amsmath,amsthm}
\usepackage{xcolor}
\usepackage[colorlinks=true,
linkcolor=webgreen,
filecolor=webbrown,
citecolor=webgreen]{hyperref}
\definecolor{webgreen}{RGB}{0,0,1}
\definecolor{recrown}{RGB}{1,.2,.6}
\usepackage{fullpage}
\usepackage{orcidlink}

\begin{document}
\newtheorem{theorem}{Theorem}
\newtheorem{corollary}[theorem]{Corollary}
\newtheorem{lemma}[theorem]{Lemma}
\theoremstyle{definition}
\newtheorem{example}{Example}
\newtheorem*{examples}{Examples}
\newtheorem*{notation}{Notation}
\newtheorem*{remark}{Remark}
\newtheorem{thmx}{Theorem}
\renewcommand{\thethmx}{\text{\Alph{thmx}}}
\newtheorem{lemmax}[thmx]{Lemma}
\renewcommand{\thelemmax}{\text{\Alph{lemmax}}}
\leftmargin=.5in
\rightmargin=0.5in
\theoremstyle{definition}
\newtheorem*{definition}{Definition}
\title[]{\bf Irreducibility via location of zeros}
\author{Jitender Singh$^{1,*}$ {\large \orcidlink{0000-0003-3706-8239}}}
\address[1]{Department of Mathematics,
Guru Nanak Dev University, Amritsar-143005, India\linebreak
 {\tt sonumaths@gmail.com}}
\author{Sanjeev Kumar$^{2}$ {\large \orcidlink{0000-0001-6882-4733}}}
\address[2]{Department of Mathematics,
SGGS College, Sector-26, Chandigarh-160019, India\linebreak
{\tt sanjeev\_kumar\_19@yahoo.co.in}}
\date{}
\parindent=0cm
\footnotetext[1]{Corresponding Author email(s): {\tt sonumaths@gmail.com; jitender.math@gndu.ac.in}\\

2010MSC: {Primary 12E05; 30C15; 11C08; 26C10; 26D05; 30A10}\\

}
\maketitle
\begin{abstract}
In this paper, we obtain several new classes of irreducible polynomials having integer coefficients whose zeros lie inside an open disk around the origin or outside a closed annular region in the complex plane. Such irreducible polynomials are devised by imposing Perron--type sufficiency conditions on their coefficients, as well as conditions on the prime factorization of the value that they assume at an integral argument.
\end{abstract}
\section{Introduction.}
The classical irreducibility criteria  due to Sch\"onemann (1846), Eisenstein (1850), Dumas (1906), and Perron (1907) have become paradigm for testing irreducibility of polynomials having rational coefficients. In fact irreducible polynomials having integer coefficients can be used to construct finite-fields, a fundamental tool in cryptography and computer science. The realm revealing riveting facts about the irreducibility of polynomials over prescribed domains has always been the cradle of such crucial classical results which for decades have witnessed cogent extensions and generalizations. Such irreducibility criteria have exhibited a close affinity to prime numbers and primality as is evident from the illustrious Buniakowski's conjecture of 1854 which asserts that if $f$ is an irreducible polynomial having integer coefficients such that the elements in the set $f(\mathbb N)$ have no common factors other than $1$, then the set $f(\mathbb N)$ contains infinitely many  prime numbers. The converse of Buniakowski's conjecture holds affirmatively via primality.

\quad Several classical irreducibility criteria for polynomials $f$ with integer coefficients rely on information on the prime factorization of $f(n)$ for some integer argument $n$. One may find the first such results in papers by St\"ackel \cite{Stackel},  P\'{o}lya and Szeg\"{o} \cite{P}, Weisner \cite{Weisner}, Ore \cite{Ore}, and Dorwart \cite{Dorwart}.
Another classical irreducibility result due to A. Cohn \cite[p.~133]{P} states that if a prime number can be expressed in base 10 as
$\sum_{i=0}^m a_i10^i$ for some positive integer $m$, then the polynomial $\sum_{i=0}^m a_iz^i$ is irreducible in $\mathbb{Z}[z]$. Cohn's result was then generalized to an arbitrary base by Brillhart et al. \cite{B} and further by Bonciocat  et al. \cite{Bonciocat2009} and \cite{Bonciocat}. In \cite{Mu}, Murty used information on the location in the complex plane of the zeros of a given polynomial as a tool to provide an elementary proof of Cohn's irreducibility criterion.
Interestingly, one of the main results of Murty \cite{Mu}, generalized by Girstmair \cite{G}  apprised of a strong converse of Buniakowski's conjecture. The irreducibility criterion of Girstmair \cite{G} was further generalized by the authors in \cite{J-S-2}. In all these results, finding information on the location of the roots of a given polynomial $f$ plays a significant role in proving its irreducibility. In this respect, Perron (1907) proved that if
$f=a_0+a_1z+\cdots+a_{m-2}z^{m-2}+a_{m-1}z^{m-1}+z^m \in \mathbb{Z}[z]$, $m\geq 2$  is such that
\begin{eqnarray*}
|a_{m-1}| &>& 1+|a_{m-2}|+|a_{m-3}|+\cdots+|a_1|+|a_0|,
\end{eqnarray*}
then $f$ is irreducible in $\mathbb{Z}[z]$. As an application of Rouch\'e's Theorem (see\cite[p.~153]{ahlfors}), the aforementioned inequality reveals that exactly $m-1$ zeros of $f$ lie in the interior of the unit disk  $|z|\leq 1$ in the complex plane and the remaining 
zero of $f$ lies outside the unit disk. The irreducibility of $f$ can be deduced as follows. On the contrary, assume that $f(z)=f_1(z)f_2(z)$ for nonconstant polynomials $f_1$ and $f_2$ in $\mathbb{Z}[z]$. Then all zeros of one of $f_1$ or $f_2$ lie in the interior of the unit circle in the complex plane. Assume that each zero $\theta$ of $f_1$ satisfies $|\theta|<1$. Then expressing $f_1$ as $f_1(z)=\prod_{\theta} (z-\theta)$, where the product is over all zeros of $f_1$, we have
\begin{eqnarray*}
1\leq |f_1(0)|=\prod_{\theta} |\theta|<1,
\end{eqnarray*}
which is a contradiction, and so, $f$ must be irreducible.  
On these lines, one has the following class of irreducible polynomials whose root location has been investigated in \cite{Miller1971} but here, we provide an alternative approach for the root location.
\begin{thmx}\label{th:A}
{\em For each natural number $m\geq 2$, the polynomial
  \begin{eqnarray*}
f(z)=z^m-z^{m-1}-\cdots-z-1
  \end{eqnarray*}
 is irreducible in $\mathbb Z[z]$.}
\end{thmx}
\begin{proof}
Using Bernoulli's inequality and the fact that $4\leq 2m\leq 2^m$, we have
\begin{eqnarray*}
\Bigl(1-\frac{1}{2^{m-1}}\Bigr)f\Bigl(2-\frac{1}{2^{m-1}}\Bigr)
= 1-2\Bigl(1-\frac{1}{2^{m}}\Bigr)^m &<& 1-2\Bigl(1-\frac{m}{2^{m}}\Bigr)
= -1+\frac{2m}{2^{m}}\leq0,
\end{eqnarray*}
which shows that $f(2-(1/2^{m-1}))<0$. Since $f(2)=1>0$, it follows that a real zero $\zeta$ of $f$ lies between $2-1/2^{m-1}$ and $2$.
If $m$ is odd, then $\zeta$ is the only real zero of $f$. If $m$ is even, then $f$ has another real zero which lies between $-1$ and $0$ since then $f(-1)=1$ and $f(0)=-1$. In either case, exactly  one real zero $\zeta$ of $f$ satisfies  $(2-1/2^{m-1})<\zeta<2$. Consequently, one has $f(z)=(z-\zeta)g(z)$, where the degree $m-1$ polynomial $g\in \mathbb{R}[z]$ is given by
\begin{eqnarray*}\label{ex2}
    g=\frac{1}{\zeta}+\Bigl(\frac{1}{\zeta}+\frac{1}{\zeta^2}\Bigr)z+\cdots+\Bigl(\frac{1}{\zeta}+\frac{1}{\zeta^2}+
    \cdots+\frac{1}{\zeta^{m-1}}\Bigr)z^{m-2}+z^{m-1}.
\end{eqnarray*}
Since $g$ is a polynomial having real positive coefficients, by an easily derivable result of Aziz and Mohammad \cite[Theorem B]{A-M}, we may observe that each  zero $\theta$ of $g$ satisfies
\begin{eqnarray*}
|\theta|\leq \max\Bigl\{1-\frac{(\zeta-1)}{\zeta^{m-1}-1},~\frac{1}{\zeta-1}\Bigl(1-\frac{1}{\zeta^{m-1}}\Bigr)\Bigr\}.
\end{eqnarray*}
Clearly, $0\leq 1-\frac{(\zeta-1)}{\zeta^{m-1}-1}<1$. Since $(2-1/2^{m-1})<\zeta<2$, we have $1-1/2^{m-1}<\zeta-1<1$, and $1-(1/\zeta^{m-1})<1-1/2^{m-1}$, and so, we get
\begin{eqnarray*}
\frac{1}{\zeta-1}\Bigl(1-\frac{1}{\zeta^{m-1}}\Bigr)<\frac{1}{1-1/2^{m-1}}\times \Bigl(1-\frac{1}{2^{m-1}}\Bigr)=1.
\end{eqnarray*}
Thus, except $\zeta$, all other zeros of $f$ lie in the interior of the unit disk $|z|<1$ in the complex plane. Now the irreducibility of $f$ is immediate.
\end{proof}
Perron \cite{Perron} also provided a multivariate version of his irreducibility criterion for polynomials over arbitrary fields. The reader may find a detailed proof of this multivariate version in \cite{Bonciocat2}. Recently, the first author and Garg proved in \cite{JRG2023} a factorization result on a class of polynomials with integer coefficients, which also generalize the irreducibility criterion of Perron, and in  \cite{JSSK2021} the authors proved some irreducibility criteria for polynomials with integers coefficients for which either the constant term or the  leading coefficient is divisible by a prime power, by using information on the location of their zeros.

\quad In this paper we shall approach the irreducibility of polynomials with integer coefficients by devising some Perron--type sufficiency conditions on their coefficients, combined with conditions on the prime factorization of their values at an integral argument. Recall that a polynomial having integer coefficients is called primitive if the greatest common divisor of its coefficients is 1. In what follows, $f^{(i)}(z)$ will as usual stand for the $i$-th derivative of $f$ with respect to $z$. Our main results are the following:
\begin{theorem}\label{th:3}
Let $f=a_0+a_1 z+\cdots+a_mz^m\in \Bbb{Z}[z]$ be primitive. Suppose there exists a positive real number $\alpha$ such that
\begin{eqnarray*}
|a_m| \alpha^m>|a_0|+|a_1|\alpha+\cdots+|a_{m-1}|\alpha^{m-1}.
\end{eqnarray*}
Further, if there exist natural numbers $n$, $d$, $k$, $\ell \leq m$, and a prime $p\nmid d$ such that
$n\geq \alpha+d$, $f(n)=\pm p^k d$, $\gcd(k,\ell)=1$, $p^k \mid \frac{f^{(i)}(n)}{i!}$ for each index $i=0, 1, \ldots,\ell-1$, and for $k>1$,  also $p\nmid \frac{f^{(\ell)}(n)}{\ell!}$, then the polynomial $f$ is irreducible in $\Bbb{Z}[z]$.
\end{theorem}
\begin{theorem}\label{th:4}
Let $f=a_0+a_1 z+\cdots+a_mz^m\in \Bbb{Z}[z]$ be primitive. Suppose there exist positive real numbers $\alpha$ and $\beta$ with $\alpha<\beta$ and an index $j\in\{0,1,\ldots,m\}$ such that
\begin{eqnarray*}
|a_j| \alpha^j>\sum_{i\neq j,~i=0}^m |a_i|\beta^i.
\end{eqnarray*}
Further, if there exist natural numbers $n$, $d$, $k$, $\ell\leq m$, and a prime $p\nmid d$ such that $\beta-d\geq n\geq \alpha+d$, $f(n)=\pm p^k d$, $\gcd(k,\ell)=1$, $p^k \mid \frac{f^{(i)}(n)}{i!}$ for each index $i=0, 1, \ldots,\ell-1$, and for $k>1$,  also $p\nmid \frac{f^{(\ell)}(n)}{\ell!}$, then the polynomial $f$ is irreducible in $\Bbb{Z}[z]$.
\end{theorem}
 A slightly different condition on
$|a_j|\alpha^j$ allows one to reach the same conclusion, as in the following result.
\begin{theorem}\label{th:5}
Let $f=a_0+a_1 z+\cdots+a_mz^m\in \Bbb{Z}[z]$ be primitive. Suppose there exist positive real numbers $\alpha$ and $\beta$ with $\alpha<\beta$ and an index $j\in\{0,1,\ldots,m\}$ such that
\begin{eqnarray*}
|a_j|\alpha^j >\bigl({\beta}/{\alpha}\bigr)^{m-j}\sum_{i\neq j,~i=0}^m |a_i|\alpha^i.
\end{eqnarray*}
Further, if there exist natural numbers $n$, $d$, $k$, $\ell\leq m$, and a prime $p\nmid d$ such that $\beta-d\geq n\geq \alpha+d$, $f(n)=\pm p^k d$, $\gcd(k,\ell)=1$, $p^k \mid \frac{f^{(i)}(n)}{i!}$ for each index $i=0, 1, \ldots,\ell-1$, and for $k>1$,  also $p\nmid \frac{f^{(\ell)}(n)}{\ell!}$, then the polynomial $f$ is irreducible in $\Bbb{Z}[z]$.
\end{theorem}
 \begin{remark}
If we take $j=m$ in Theorem \ref{th:5}, then the condition on the coefficients of $f$ becomes $|a_m|\alpha^m >\sum_{i\neq j,~i=0}^{m} |a_i|\alpha^i$, which holds  for all $\beta>\alpha$. In view of this, once an $n$ with $n\geq \alpha+d$ is found for which $|f(n)|/d$ is prime or a prime-power, we can take $\beta=n+d$ so that $\beta-d\geq n$ can always be satisfied. Thus, for $j=m$,  Theorem \ref{th:5} reduces to Theorem \ref{th:3}.
\end{remark}
We mention here the following three important particular cases of Theorems \ref{th:3}-\ref{th:5}, respectively.
\begin{corollary}\label{th:1}
   Let $f=a_0+ a_{1}z+\cdots+a_m z^m\in \Bbb{Z}[z]$ be  primitive. Suppose there exists a positive real number $\alpha$ such that
\begin{eqnarray*}
|a_m| \alpha^m>|a_0|+|a_1|\alpha+\cdots+|a_{m-1}|\alpha^{m-1}.
\end{eqnarray*}
If there exist natural numbers $n$ and $d$ satisfying $n\geq \alpha+ d$ for which  $|f(n)|/d$ is prime, or $|f(n)|/d$ is a prime-power coprime to $|f'(n)|$, then $f$ is irreducible in $\mathbb{Z}[z]$.
 \end{corollary}
\begin{corollary}\label{th:2}
    Let $f=a_0+ a_{1}z+\cdots+a_m z^m\in \Bbb{Z}[z]$ be primitive. Suppose there exist positive real numbers $\alpha$ and $\beta$ with $\alpha<\beta$ and an index $j\in \{0,1,\ldots,m\}$ such that
\begin{eqnarray*}
|a_j| \alpha^j>\sum_{i\neq j,~i=0}^{m} |a_i|\beta^{i}.
\end{eqnarray*}
If there exist natural numbers $n$ and $d$ satisfying $\beta-d\geq n\geq \alpha+d$ for which  $|f(n)|/d$ is prime, or $|f(n)|/d$ is a prime-power coprime to $|f'(n)|$, then $f$ is irreducible in $\mathbb{Z}[z]$.
 \end{corollary}
 \begin{corollary}\label{th:2b}
      Let $f=a_0+ a_{1}z+\cdots+a_m z^m\in \Bbb{Z}[z]$ be primitive. Suppose there exist  positive real numbers  $\alpha$ and $\beta$  with $\alpha<\beta$ and an index $j\in \{0,1,\ldots,m\}$ such that
\begin{eqnarray*}
|a_j|\alpha^j >\bigl({\beta}/{\alpha}\bigr)^{m-j}\sum_{i\neq j,~i=0}^{m} |a_i|\alpha^i.
\end{eqnarray*}
If there exist natural numbers $n$ and $d$ satisfying $\beta-d\geq n\geq \alpha+d$ for which  $|f(n)|/d$ is prime, or $|f(n)|/d$ is a prime-power coprime to $|f'(n)|$, then  $f$ is irreducible in $\mathbb{Z}[z]$.
 \end{corollary}
For some general results of this type for polynomials that have no roots in a given annular region we refer the reader to \cite[Lemma 1.1]{Bonciocat2009} and \cite[Lemma 2.1]{Bonciocat}. We will give in the last section of the paper a series of examples of polynomials whose irreducibility follows by our main results.
\section{Proofs of Theorems \ref{th:3}-\ref{th:5}}
To prove Theorems \ref{th:3}--\ref{th:5}, we will employ the techniques  used in \cite{Weisner},\cite{Bonciocat2009} and \cite{Bonciocat}. We will also need the following result proved in \cite[Lemma 3]{JSSK2021}:
\begin{lemmax}[Singh and Kumar \cite{JSSK2021}\label{L1}]
\em Let $f=a_0+ a_{1}z+\cdots+a_m z^m$; $f_1$ and  $f_2$ be nonconstant polynomials in $\Bbb{Z}[z]$ such that $f(z)=f_1(z)f_2(z)$. Suppose there exists a prime $p,$ coprime natural numbers $k\geq 2$ and $\ell\leq m$ such that $p^k\mid \gcd(a_0,a_1,\ldots,a_{\ell-1})$ and $p^{k+1}\nmid a_0$. If $p\mid f_1(0)$ and $p\mid f_2(0)$, then $p\mid a_\ell$.
 \end{lemmax}
For each $i=0,\ldots,m$, let $s_i=\frac{f^{(i)}(n)}{i!}$, and define
\begin{equation}\label{a2}
g(z)=f(z+n)=s_0+s_1z+\cdots+s_mz^m.
\end{equation}
In order to prove our results, 
it will be sufficient to show that $g$ is irreducible. According to 
the hypothesis of each of Theorems \ref{th:3} - \ref{th:5}, $p$ divides $s_i$ for each $i=0,\ldots,\ell-1$. Suppose on the contrary that $g(z)=g_1(z)g_2(z)$ for two nonconstant polynomials $g_1$ and $g_2$ in $\Bbb{Z}[z]$. In particular, this implies that
\begin{equation}\label{a4}
|g(0)|=|f(n)|=p^k d=|g_1(0)||g_2(0)|.
\end{equation}
Therefore, $p$ divides at least one of the factors $|g_1(0)|$ or $|g_2(0)|$. Assume without loss of generality that $p$ divides $|g_2(0)|$.
 \begin{proof}[\bf Proof of Theorem \ref{th:3}]
According to our hypothesis we have $|a_m| \alpha^m>\sum_{j=0}^{m-1}|a_j|\alpha^j$, so for a complex number $z$ with $|z|\geq \alpha$ we deduce that
\begin{eqnarray*}
|f(z)|\geq |z|^m\Bigl(|a_m|-\sum_{i=0}^{m-1}|a_i||z|^{-(m-i)}\Bigr)\geq \alpha^m\Bigl(|a_m|-\sum_{i=0}^{m-1}|a_i|\alpha^{-(m-i)}\Bigr)>0.
\end{eqnarray*}
This shows that each zero $\theta$ of $f$ must satisfy the inequality $|\theta|<\alpha$. We mention here that this is also a direct consequence of Rouch\'e' s Theorem.

\quad Consider first the case that $p$ also divides $|g_1(0)|$. Then $k>1$, $p$ divides each of $|g_1(0)|$ and $|g_2(0)|$, and $p$ divides $s_i$ for each $i=0,\ldots,\ell-1$. Since $p\nmid d$, it follows that $p^{k+1}\nmid |g(0)|$. We may now apply Lemma \ref{L1} to $g$, and conclude that $p\mid s_\ell$, which contradicts our hypothesis.

\quad Now consider the case when $p\nmid |g_1(0)|$, which implies that $p^k\mid g_2(0)$. From \eqref{a4}, we then find that $|g_1(0)|$ divides $d$, so in particular we have $|g_1(0)|\leq d$. If $c\neq 0$ is the leading coefficient of $g_1$, then we may write $g_1(z)=c\prod_{\theta}(z-\theta)$, with the product running over all the zeros $\theta$ of $g_1$.
If $g(\theta)=0$, then from \eqref{a2} we have $f(\theta+n)=0$, so $\theta +n$ must be a root of $f$, and hence we must have $|\theta+n|<\alpha$. Consequently, we have
\begin{eqnarray*}
|\theta| &=& |n-(n+\theta)|\geq n-|n+\theta|>n-\alpha\geq d,
\end{eqnarray*}
which yields $d\geq |g_1(0)|=|c|\prod_{\theta}|\theta|>|c|d^{\deg{g_1}}\geq |c|d\geq d$, a contradiction. This completes the proof of the theorem.
\end{proof}
\begin{proof}[\bf Proof of Theorem \ref{th:4}]
If $z\in \mathbb{C}$ is such that $\alpha\leq |z|\leq \beta$, then using our hypothesis on $|a_j|\alpha ^j$ we obtain
\begin{eqnarray*}
|f(z)|\geq |a_j||z|^{j}-\sum_{i\neq j,~i=0}^{m}|a_i||z|^i\geq |a_j|\alpha^{j}-\sum_{i\neq j,~i=0}^{m}|a_i|\beta^i>0,
\end{eqnarray*}
which shows that each zero $\theta$ of $f$ satisfies $|\theta|<\alpha$ or $\beta<|\theta|$.

\quad As in the preceding proof, if $p$ also divides $|g_1(0)|$, then $k>1$ and by applying Lemma \ref{L1} to $g$, we deduce again that $p\mid s_\ell$, which contradicts our hypothesis.

\quad Now consider the case when $|g_1(0)|$ is not divisible by $p$. From \eqref{a4}, we deduce again that $|g_1(0)|\leq d$.  If $g(\theta)=0$, then by \eqref{a2} we see that $\theta +n$ must be a root of $f$, and hence by the above conclusion on the roots of $f$ we either have $|\theta+n|<\alpha$, or $\beta< |\theta+n|$. As a consequence, we obtain the following inequalities.
\begin{eqnarray*}
|\theta| &=& |n-(n+\theta)|\geq n-|n+\theta|>n-\alpha\geq d,~\text{if}~|\theta+n|<\alpha,\\
|\theta| &=& |n+\theta-n|\geq |n+\theta|-n>\beta-n\geq d,~\text{if}~\beta<|\theta+n|.
\end{eqnarray*}
Thus $|\theta|>d$ for each zero $\theta$ of $g$, and in particular for each zero $\theta$ of $g_1$. Thus, if we write as before $g_1(z)=c\prod_{\theta}(z-\theta)$, we obtain
\[
d\geq |g_1(0)|=|c|\prod_{\theta}|\theta|>|c|d^{\deg{g_1}}\geq d,
\]
which is a contradiction. This finishes the proof.
\end{proof}
\begin{proof}[\bf Proof of Theorem \ref{th:5}]
Here we will prove that the inequality
\[
|a_j|\alpha^j >(\beta/\alpha)^{m-j}\sum_{i\neq j,~i=0}^m |a_i|\alpha^i
\]
forces the zeros $\theta$ of $f$ to satisfy $|\theta|<\alpha$, or $\beta<|\theta|$. Indeed, if $z\in \mathbb{C}$ is such that $\alpha\leq |z|\leq \beta$, then we have $\frac{1}{|z|^{m-j}}\geq \frac{1}{\beta^{m-j}}$ and $\frac{-1}{|z|^{m-i}}\geq \frac{-1}{\alpha^{m-i}}$. Using these inequalities along with our hypothesis on $|a_j|\alpha^j$, one obtains
\begin{eqnarray*}
|f(z)|\geq |z|^m\Bigl(\frac{|a_j|}{|z|^{m-j}}-\sum_{i\neq j,~i=0}^{m}\frac{|a_i|}{|z|^{m-i}}\Bigr)&\geq& \alpha^m\Bigl(\frac{|a_j|}{\beta^{m-j}}-\sum_{i\neq j,~i=0}^{m}\frac{|a_i|}{\alpha^{m-i}}\Bigr),\\
&=& \Bigl(|a_j|\alpha^j\Bigl(\frac{\alpha}{\beta}\Bigr)^{m-j}-\sum_{i\neq j,~i=0}^{m}|a_i|\alpha^i\Bigr)>0,
\end{eqnarray*}
which shows that each zero $\theta$ of $f$ satisfies $|\theta|<\alpha$ or $\beta<|\theta|$. The rest of the proof follows as in Theorem \ref{th:4}.
\end{proof}
\begin{remark}
By using a Newton polygon argument, Bonciocat et al. \cite[Lemma 1.4]{Bonciocat2013} proved the following lemma, which turns out to be equivalent to Lemma \ref{L1}, that was independently proved later by the authors in \cite{JSSK2021}.
\begin{lemmax}[Bonciocat et al. \cite{Bonciocat2013}]\label{L2}
\em Let $f, g\in \mathbb{Z}[X]$ be two polynomials with $\deg g=n$ and $\deg f=n-d$, $d\geq 1$. Let also $p$ be a prime number that divides none of the leading coefficients of $f$ and $g$, and let $k$ be any positive integer prime to $d$. If $f(X) + p^kg(X)$ may be written as a product of two non-constant polynomials with integer coefficients, say $f_1$ and $f_2$, then one of the leading coefficients of $f_1$ and $f_2$ must be divisible by $p^k$.
\end{lemmax}
We show that Lemma \ref{L1} is  Lemma \ref{L2} in disguise, written for the reciprocal of $f$. To see this equivalence, we will first prove that Lemma \ref{L2} implies Lemma \ref{L1}. We proceed as follows.

Let $f(z)=f_1(z)f_2(z)$ be as in Lemma \ref{L1}, and let $\tilde{f}(z)=z^mf(1/z)=b_0+b_1z+\cdots +b_mz^m$, with $b_i=a_{m-i}$ for $i=0,1,\dots ,m$. Thus $p^k\mid \gcd (b_{m-\ell +1},b_{m-\ell +2},\dots ,b_m)$, $p^{k+1}\nmid b_m$ and $\tilde{f}(z)=\tilde{f_1}(z)\tilde{f_2}(z)$. Besides, the leading coefficient of $\tilde{f_1}(z)$ is $f_1(0)$ and the leading coefficient of $\tilde{f_2}(z)$ is $f_2(0)$. To prove Lemma \ref{L1}, assume on the contrary that $p\nmid a_{\ell }$. Thus $p\nmid b_{m-\ell }$. If we write
\begin{eqnarray*}
b_0+b_1z+\cdots +b_mz^m=b_0+b_1z+\cdots +b_{m-\ell }z^{m-\ell }+p^k(b'_{m-\ell +1}z^{m-\ell +1}+\cdots +b'_mz^m)
\end{eqnarray*}
with $b'_i=\frac{b_i}{p^k}\in \mathbb{Z}$ for $i=m-\ell +1,\dots ,m$ (since $p^k\mid \gcd (b_{m-\ell +1},b_{m-\ell +2},\dots ,b_m)$), we actually may write $\tilde{f}(z)=F(z)+p^kG(z)$ with
\begin{eqnarray*}
F(z) & = & b_0+b_1z+\cdots +b_{m-\ell }z^{m-\ell } \quad \quad {\rm and}\\
G(z) & = & b'_{m-\ell +1}z^{m-\ell +1}+\cdots +b'_mz^m,
\end{eqnarray*}
with $p\nmid b_{m-\ell}$ and $p\nmid b'_m$. We may apply now Lemma \ref{L2} with $d=\ell $ to the polynomial $F+p^kG$ and conclude that one of the leading coefficients of $\tilde{f_1}(z)$ and $\tilde{f_2}(z)$ must be divisible by $p^k$, so the other one can not be divisible by $p$. This contradicts the assumption that both $f_1(0)$ and $f_2(0)$ are divisible by $p$.

Conversely, assume that Lemma \ref{L1} holds, so if we state it for the reciprocal of $f$, we have:
\begin{lemmax}[Singh and Kumar \cite{JSSK2021}\label{L3}]
\em Let $f = a_0+a_1z+\cdots+a_mz^m$, $f_1$ and $f_2$  be nonconstant polynomials in $\mathbb{Z}[z]$ such that $f(z)=f_1(z)f_2(z)$. Suppose there exists a prime $p$, coprime
natural numbers $k\geq 2$ and $\ell\leq m$ such that $p^k\mid \gcd(a_m, a_{m-1},\dots, a_{m-\ell+1})$ and $p^{k+1}\nmid a_m$. If the leading coefficients of $f_1$ and $f_2$ are both divisible by $p$, then $p\mid a_{m-\ell }$.
\end{lemmax}
We will prove now that Lemma \ref{L3} implies Lemma \ref{L2}. So let $f$ and $g$ be as in Lemma \ref{L2}. Assume that $f(z)=b_0+b_1z+\cdots +b_{n-d}z^{n-d}$ and $g(z)=c_0+c_1z+\cdots +c_nz^n$, say, with $p\nmid b_{n-d}c_n$. Then we may write
\begin{eqnarray*}
f+p^kg & = & a_0+a_1z+\cdots +a_nz^n\\
& = & (b_0+p^kc_0)+(b_1+p^kc_1)z+\cdots +(b_{n-d}+p^kc_{n-d})z^{n-d}\\
&  & +p^kc_{n-d+1}z^{n-d+1}+\cdots +p^kc_nz^n.
\end{eqnarray*}
Now assume on the contrary that $f(z) + p^kg(z)=f_1(z)f_2(z)$ and that both leading coefficients of $f_1$ and $f_2$ are multiples of $p$. By Lemma \ref{L3} with $n$ instead of $m$ and $d$ instead of $\ell $ we conclude that $p\mid a_{n-d}$, that is $p\mid (b_{n-d}+p^kc_{n-d})$, which implies that $p\mid b_{n-d}$, in contradiction to the assumption of Lemma \ref{L2} that $p$ divides none of the leading coefficients of $f$ and $g$.

Therefore, Lemma \ref{L1} and Lemma \ref{L2} are equivalent.
\end{remark}
\section{Examples}
We will give now some examples of classes of polynomials whose irreducibility may be proved by using Theorems \ref{th:3}-\ref{th:5}.
\begin{example}
Let $F_1$ be the polynomial defined by
\begin{equation*}
F_1=p^\ell+p^{\ell-1}z+p^{\ell-2}z^2+\ldots+pz^{\ell-1}-\ell z^{\ell}\pm (p^{k} d)z^m,
\end{equation*}
where $\ell$, $k\geq m+\ell$, $m>\ell$, and $d$ are positive integers, and $p\geq 1+d$ is a prime number with $\gcd(p,\ell)=1$, $\gcd(k,\ell)=1$. Here,  $a_i=p^{\ell-i}$ for $i=0,\ldots,\ell-1$, $a_\ell=-\ell$, $a_i=0$ for $i=\ell+1,\ldots,m-1$, and $a_m=\pm p^k d$. Taking $\alpha=1$, $n=p$, we have
\begin{eqnarray*}
|a_m|\alpha^m=p^{k}d\geq p^{m+\ell+1}&=&\underbrace{p^m+p^m+\cdots+p^m }_{p^{\ell+1}\text{-times}}\\
&>& (p^\ell+p^{\ell-1}+p^{\ell-2}+\cdots+p^1)+\ell=\sum_{i=0}^{m-1}|a_i|\alpha^i,
\end{eqnarray*}
where $n=p\geq 1+d=\alpha+d$. Since $\gcd(p,\ell)=1$, the polynomial $F_1$ is primitive, and  $|F_1(p)|/d=p^{k+m}$, where
 $p$ divides ${F_1}^{(i)}(p)/i!$ for $i=0,\ldots,\ell-1$, and ${F_1}^{(\ell)}(p)/\ell! =-\ell$, which is assumed to be coprime to $p$. By Theorem \ref{th:3}, the polynomial $F_1$ is irreducible in $\mathbb{Z}[z]$.
\end{example}
\begin{example}
Consider the polynomial
\begin{equation*}
F_2=p^m +(p^{\ell-1}z+p^{\ell-2}z^2+\ldots+pz^{\ell-1})-(\ell-1)z^{\ell}\pm (p^{k-j} d)z^j-z^m,
\end{equation*}
where $\ell\geq 2$, $j>\ell$, $k\geq 2(m+\ell)+j$, $m\geq j+1$ and $d$ are positive integers, and $p\geq 1+d$ is a prime number with $\gcd(p,\ell-1)=1$, $\gcd(k,\ell)=1$. Here, $a_0=p^m$, $a_i=p^{\ell-i}$ for $i=1,\ldots,\ell-1$, $a_\ell=-(\ell-1)$, $a_i=0$ for $i\neq j$ and $i\in \{\ell+1,\ldots,m-1\}$,  $a_j=\pm p^{k-j}d$, and $a_m=-1$. Taking $\alpha=1/p$, $\beta=p^2$, and $n=p$, we have
\begin{eqnarray*}
|a_j|\alpha^j=p^{k}d\geq p^{2m+2\ell}&=&\underbrace{p^{2m}+p^{2m}+\cdots+p^{2m}}_{p^{2\ell}\text{-times}}\\
&>& p^m+(p^{\ell+1}+p^{\ell+2}+\cdots+p^{2\ell-1})+(\ell-1)p^{2\ell}+p^{2m}\\
&=&\sum_{i\neq j,~i=0}^{m}|a_i|\beta^i,
\end{eqnarray*}
where $\beta-d=p^2-d>p=n\geq 1+d>\alpha+d$. Further, the polynomial $F_2$ is primitive, and  $|F_2(p)|/d=p^{k}$, where
 $p$ divides ${F_2}^{(i)}(p)/i!$ for $i=0,\ldots,\ell-1$, and ${F_2}^{(\ell)}(p)/\ell! =-(\ell-1)$, which is assumed to be coprime to $p$. So, by Theorem \ref{th:4}, the polynomial $F_2$ is irreducible in $\mathbb{Z}[z]$.
\end{example}
 \begin{example}
For $k\geq m+2\geq 4$ and prime $p$ with $n=p^2>p+d$, $1\leq d\leq p(p-1)$;  the polynomial
    \begin{eqnarray*}
F_3=-p^2+z\pm p^{k-m}dz^m
    \end{eqnarray*}
  satisfies the hypothesis of Corollary \ref{th:1} with
  \begin{eqnarray*}
  \alpha&=&p,~a_0=-p^2,~a_1=1, a_i=0~(i=2,3,\ldots,m-1),~a_m=\pm p^{k-m} d,\\
  n&=&p^2\geq p+d=\alpha+d,~F_3(p^2)=\pm p^{k+m} d,~F_3'(p^2)\equiv 1\mod p,
  \end{eqnarray*}
so that $\gcd(|F_3(p^2)|/d,|F_3'(p^2)|)=1$, and
  \begin{eqnarray*}
|a_m|\alpha^m=p^{k}d\geq p^3>p^2+p=\sum_{i=0}^{m-1}|a_i|\alpha^i.
    \end{eqnarray*}
So, the polynomial $F_3$ is irreducible in $\mathbb{Z}[z]$.
 \end{example}
 \begin{example}
Consider the polynomial
    \begin{eqnarray*}
F_4=(z-p)+(z-p)^2+\cdots+(z-p)^{m-1}\pm (p^{2k-2}d)z^m,
    \end{eqnarray*}
for $k\geq m\geq 2$ and $p\geq 1+d$. Here, $a_i=\sum_{j=i}^{m-1} {j\choose i}(-p)^{j-i}$ for $i=0,1,\ldots, m-1$. Further, $a_m=\pm p^{2k-2} d$, $\alpha=1$, and  $n=p\geq 1+d$. We find that $F_4(p)=\pm p^{2k+m-2} d$, $F_4'(p)\equiv 1\mod p$. These along with the fact that $p^2>1+p$ yield the following:
  \begin{eqnarray*}
|a_m|\alpha^m=\frac{p^{2k}d}{p^2}\geq \frac{(p^2)^m d}{p^2}>\frac{(1+p^2)^m}{p^2}>(1+p^2)\frac{(1+p^2)^{m-1}-1}{1+p^2-1}>\sum_{i=0}^{m-1}|a_i|\alpha^i.
    \end{eqnarray*}
Since $a_{m-1}=1$, it follows that $F_4$ is a primitive polynomial. By Corollary \ref{th:1}, the polynomial  $F_4$ is irreducible in $\mathbb{Z}[z]$.
 \end{example}
\begin{example}
Consider the polynomial
\begin{equation*}
F_5=(p^m-p)+z\pm (p^{1+2m+j} d)z^j-z^m,~1\leq d\leq p-1,~2\leq j\leq m-1,
\end{equation*}
where $j$, $m\geq 3$ and $d$ are positive integers and $p$ is a prime number. Here, $a_0=p^m-p$, $a_1=1$, $a_i=0$ for $i\not\in\{0,1,j,m\}$,  $a_j=\pm p^{1+2m+j}d$, and $a_m=-1$. Taking $\alpha=1/p$, $\beta=p^2$, and $n=p$, we have
\begin{eqnarray*}
|a_j|\alpha^j&=&p^{1+2m}d> p^m-p+p^2+p^{2m}=\sum_{i\neq j,~i=0}^{m}|a_i|\beta^i,\\
\alpha+d&<& 1+d\leq n=p< p^2-d=\beta-d.
\end{eqnarray*}
Further, the polynomial $F_5$ is primitive, and  $|F_5(p)|/d=p^{1+2(m+j)}$ and ${F_5}'(p)\equiv 1\mod p$. By Corollary \ref{th:2}, the polynomial $F_5$ is irreducible in $\mathbb{Z}[z]$.
\end{example}
\begin{example}
Consider the polynomial
\begin{equation*}
F_6=(p^m-p)+z\pm (p^{k} d)z^j-z^m;~1\leq d\leq p-1,~2\leq j\leq m-1,
\end{equation*}
where $k\geq 3m-2j$, $m\geq 3$, $j$, and $d$ are positive integers and $p$ is a prime number. Here, $a_0=p^m-p$, $a_1=1$, $a_i=0$ for $i\not\in\{0,1,j,m\}$,  $a_j=\pm p^{k}d$, and $a_m=-1$. Taking $\alpha=1$, $\beta=p+d$, and $n=p$, we have
\begin{eqnarray*}
|a_j|\alpha^j&=&p^{k}d\geq p^{3m-2j}>(p+d)^{m-j}(p^m-p+2)=\bigl({\beta}/{\alpha}\bigr)^{m-j}\sum_{i\neq j,~i=0}^{m}|a_i|\alpha^i,\\
\alpha+d&=& 1+d\leq n=p=\beta-d.
\end{eqnarray*}
Further, the polynomial $F_6$ is primitive, and  $|F_6(p)|/d=p^{k+j}$ and ${F_6}'(p)\equiv 1\mod p$. By Corollary \ref{th:2b}, the polynomial $F_6$ is irreducible in $\mathbb{Z}[z]$.
\end{example}
\subsection*{Disclosure statement}
The authors report there are no competing interests to declare.
\bibliographystyle{ams}

\end{document}